\documentclass[12pt]{amsart}

\usepackage{graphicx}
\usepackage{amsmath}
\usepackage{amscd}
\usepackage{amsfonts}
\usepackage{amssymb}
\usepackage{color}
\usepackage{mathtools}
\usepackage[pagebackref,hypertexnames=false, colorlinks, citecolor=red,linkcolor=blue, urlcolor=red]{hyperref}

\numberwithin{equation}{section}

\newtheorem{theorem}{Theorem}[section]
\newtheorem{lemma}[theorem]{Lemma}
\newtheorem{proposition}[theorem]{Proposition}

\newtheorem{definition}[theorem]{Definition}
\newtheorem{corollary}[theorem]{Corollary}

\theoremstyle{definition}
\newtheorem{example}[theorem]{Example}
\newtheorem{remark}[theorem]{Remark}

\DeclarePairedDelimiter{\floor}{\lfloor}{\rfloor}
\newcommand{\be}{\begin{equation}}
\newcommand{\ee}{\end{equation}}
\newcommand{\bes}{\begin{equation*}}
\newcommand{\ees}{\end{equation*}}

\newcommand{\cH}{\mathcal{H}}

\newcommand{\cL}{\mathcal{L}}

\newcommand{\mb}[1]{\mathbb{#1}}

\newcommand{\bC}{\mathbb{C}}

\newcommand{\ol}{\overline}

\begin{document}

\title[Stable division and essential normality]{Stable division and essential normality: the non-homogeneous and quasi homogeneous cases}

\author{Shibananda Biswas}
\address{Department of Mathematics and Statistics, Indian Institute of Science Education and Research Kolkata, Mohanpur - 741246, West Bengal, India
}
\email{shibananda@gmail.com}

\author{Orr Moshe Shalit}
\address{Faculty of Mathematics\\
Technion - Israel Institute of Mathematics\\
Haifa\; 3200003\\
Israel}
\email{oshalit@tx.technion.ac.il}

\thanks{The work of S. Biswas is partially supported by Inspire Faculty Fellowship (IFA-11MA-06)
funded by DST, India.
The work of O.M. Shalit is partially supported by ISF Grant no. 474/12, and by
EU FP7/2007-2013 Grant no. 321749}
\begin{abstract}
Let $\cH_d^{(t)}$ ($t \geq -d$, $t>-3$) be the reproducing kernel Hilbert space on the unit ball $\mb{B}_d$ with kernel
\[
k(z,w) = \frac{1}{(1-\langle z, w \rangle)^{d+t+1}} .
\]
We prove that if an ideal $I \triangleleft \mb{C}[z_1, \ldots, z_d]$ (not necessarily homogeneous) has what we call the  {\em approximate stable division property}, then the closure of $I$ in $\cH_d^{(t)}$ is $p$-essentially normal for all $p>d$.
We then show that all quasi homogeneous ideals in two variables have the stable division property, and combine these two results to obtain a new proof of the fact that the closure of any quasi homogeneous ideal in $\mb{C}[x,y]$ is $p$-essentially normal for $p>2$.
\end{abstract}

\maketitle

\section{Introduction and preliminaries}
\subsection{The basic setting}

Let $\mb{C}[z] = \mb{C}[z_1, \ldots, z_d]$ denote the ring of complex polynomials in $d$ variables.
Every $d$-tuple $T = (T_1, \ldots, T_d)$ of commuting operators on a Hilbert space $H$ defines an action of $\mb{C}[z]$ on $H$ via
\[
p\cdot h = p(T) h \quad, \quad p \in \mb{C}[z], h \in H.
\]
This gives $H$ the structure of a module over $\mb{C}[z]$, and we say that $H$ is a {\em Hilbert module} over $\mb{C}[z]$. The Hilbert module $H$ is said to be {\em essentially normal} if the commutator $[T_i, T_j^*]:= T_i T_j^* - T_j^* T_i$ is compact for all $i,j$.
If $p \geq 1$ and $[T_i,T_j^*] \in \cL^p$ (meaning that $|[T_i, T_j^*]|^p$ is trace class), then we say that $H$ is {\em $p$-essentially normal}.

One of the interesting ways in which Hilbert modules arise is as follows. Assume that an inner product is given on $\mb{C}[z]$ such that the multiplication operators $S_1, \ldots, S_d$ given by
\[
(S_i f)(z) = z_i f(z)
\]
are bounded. Let $H$ denote the completion of $\mb{C}[z]$ with respect to the given inner product.
Then $H$ becomes a Hilbert module. Examples includes all analytic Hilbert modules on bounded domain in $\mathbb C^d$ \cite{CG}. In the case where (1) monomials of different degrees are orthogonal, and (2) the row operator $[S_1, \ldots, S_d] : H^{(d)} \rightarrow H$ has closed range, these types of modules were referred to as {\em graded completions} of $\mb{C}[z]$ in \cite{Arv07}.
Moreover, if $r \in \mb{N}$, then $H \otimes \mb{C}^r$ is also a Hilbert module in a natural way --- such modules were referred to as {\em standard Hilbert modules} in \cite{Arv07}. Finally, if $M$ is a closed submodule of $H \otimes \mb{C}^r$, then $M$ and $(H \otimes \bC^r)/M$ can also be given a natural Hilbert module structure.

\subsection{Homogeneous and quasi homogeneous submodules}

The algebra of polynomials has a natural grading by degree
\[
\mb{C}[z] = \mb{H}_0 \oplus \mb{H}_1 \oplus \mb{H}_2 \oplus \ldots.
\]
This induces a direct sum decomposition
\[
H = \mb{H}_0 \oplus \mb{H}_1 \oplus \mb{H}_2 \oplus \ldots .
\]
A vector valued polynomial $h \in H \otimes \mb{C}^r$ is said to be {\em homogeneous} if $h \in \mb{H}_n \otimes \mb{C}^r$ for some $n$. A module is said to be {\em homogeneous} if it is generated by homogeneous polynomials.

Let ${\bf n} = (n_1, \ldots, n_d) \in \mb{N}^d$, where $n_i > 0$ for all $i$. A vector valued polynomial $h \in \mb{C}[z] \otimes \mb{C}^r$ is said to be {\em ${\bf n}$-quasi homogeneous of degree $m$}, denoted $deg_{\bf n}(h) = m$, if for every monomial $z^\alpha \otimes v_\alpha$ appearing in $h$ it holds that $n_1 \alpha_1 + \ldots n_d \alpha_d = m$.
A module $M \subseteq \mb{C}[z] \otimes \mb{C}^r$ is said to be {\em ${\bf n}$-quasi homogeneous} if it is generated by ${\bf n}$-quasi homogeneous polynomials.

Any ${\bf n}$-quasi homogeneous submodule of a graded completion $H$ (including $H$ itself) also decomposes as a direct sum of ${\bf n}$-quasi homogeneous summands.

\subsection{Stable division and approximate stable division}

\begin{definition}
A submodule $M \triangleleft \mb{C}[z] \otimes \mb{C}^r\subseteq H\otimes\mathbb C^r$ is said to have the {\em approximate stable division property} if there are elements $f_1, \ldots, f_k \in M$ and a constant $A$ such that for every $\epsilon > 0$ and for every $h \in M$, one can find polynomials $g_1, \ldots, g_k \in \mb{C}[z]$ such that
\be\label{eq:sum_app}
\|h - \sum_{i=1}^k g_i f_i\| \leq \epsilon ,
\ee
together with the norm constraint
\be\label{eq:stab_app}
\sum_{i=1}^k \|g_i f_i \| \leq A \|h\| .
\ee
If $\epsilon$ can be chosen $0$, then we say that $M$ has the {\em stable division property}. The set $\{f_1, \ldots, f_k\}$ is said to be an {\em (approximate) stable generating set}.
\end{definition}

\subsection{The Hilbert spaces $\cH_d^{(t)}$}

In this paper we will let $\cH_d^{(t)}$ ($t \geq -d$) be the reproducing kernel Hilbert space on the unit ball $\mb{B}_d$ with kernel
\[
k(z,w) = \frac{1}{(1-\langle z, w \rangle)^{d+t+1}} .
\]
This is the completion of $\mb{C}[z]$ with respect to the inner product making all monomials orthogonal, and for which
\be\label{eq:tnorm}
\|z^\alpha\|_t^2 = \frac{\alpha!}{\Pi_{i=1}^{|\alpha|} (d+t+i)} .
\ee
When $t = 0$ and $t=-1$, respectively, we get the Bergman space and the Hardy space, respectively, on the unit ball (see \cite{ZhuBook}).
When $t = -d$ then we get the {\em Drury-Arveson space}, denoted $H^2_d$ (see \cite{Arv98, ShalitSurvey}).
Note that if $|\alpha| = n$, then
\be\label{eq:quotient_norm}
\frac{\|z^\alpha\|_t^2}{\|z^\alpha\|_{H^2_d}^2} = \frac{n!}{\Pi_{i=1}^{n} (d+t+i)} = \frac{\Gamma(n+1)\Gamma(d+t+1)}{\Gamma(d+t+n+1)} .
\ee
Thus for all $f \in \mb{H}_n$ we have
\[
\|f\|^2_t = c_{n,t}\|f\|^2_{H^2_d} ,
\]
where $c_{n,t}$ denotes the right hand side of (\ref{eq:quotient_norm}).
We see that the $c_{n,t}$ are less than one, decreasing, and tend to $0$ with $n\rightarrow \infty$ for all $t > -d$.
In fact, Arveson showed in the appendix of \cite{Arv07} that if $H$ is any standard Hilbert module which has {\em maximal symmetry} (in the sense that the inner product in $H$ is invariant under the action of the unitary group of $\mb{C}^d$), then any homogeneous polynomial $f \in \mb{H}_n$ has norm
\[
\|f\|^2_H = \gamma_n\|f\|^2_{H^2_d}.
\]
If $f = \sum_{i=k}^m f_i$ is a polynomial with homogeneous parts $f_k, \ldots, f_m$, then
\[
\|f\|_H^2 = \sum_{i=k}^m \|f_i\|_H^2 \leq \sum_{i=k}^m \gamma_i\|f_i\|_{H_d^2}^2 \leq \max_{k \leq i \leq m}\gamma_i \|f\|_{H^2_d}^2 .
\]
Likewise $\|f\|^2_H \geq \min_{k \leq i \leq m} \gamma_i \|f\|^2_{H^2_d}$.
For the spaces $\cH_d^{(t)}$ we get
\be\label{eq:equiv_norm}
c_{m,t}\|f\|^2_{H^2_d} \leq \|f\|^2_t \leq c_{k,t}\|f\|^2_{H^2_d} \leq \|f\|^2_{H^2_d}.
\ee
\subsection{What this paper is about}

Arveson conjectured \cite{Arv02} that every quotient of $H^2_d \otimes \bC^r$ by a homogneous submodule $M$ is $p$-essentially normal for all $p>d$; later Douglas refined this conjecture \cite{Dou06b}   to include the range $p > \dim (M)$.
The Arveson-Douglas conjecture drew several mathematicians to work on essential normality (see, e.g., \cite{Arv05,Arv07,Dou06a,DS11,Esc11,GW08,Ken15,KenSha12,KenSha15,Sha11}), and positive results showing that the conjecture holds for certain classes of ideals were obtained.

With time it has come to be believed that the {\em homogeneity} assumption should not be of central importance.
Moreover, the conjecture is believed to hold for $\cH_d^{(t)}$ for all $t \geq -d$.
In fact, it is a folklore result that in the homogeneous setting the validity of the conjecture in one of these spaces is equivalent to its validity in all these spaces.
In two important papers \cite{DW11} and \cite{FX13} it was proved that the closure in some of the spaces $\cH_d^{(t)}$ of a principal ideal $I$ is $p$-essentially normal for all $p>d$ (see also \cite{FX15} and \cite{GZ13} which followed).
An even bigger breakthrough occurred with the appearance of the two papers \cite{DTY14} and \cite{EngEsc13}, in which the conjecture was confirmed under natural smoothness and transversality conditions.

In a different effort, the paper \cite{Sha11} suggested an approach that was based on stable division; in that paper it was shown that if a homogeneous submodule $M \triangleleft \bC[z] \otimes \bC^r \subset H^2_d \otimes \bC^r$ has the stable division property, then it fully satisfies the Arveson-Douglas conjecture.
This was used to obtain new proofs of the conjecture (and also to give some explanation of it) for classes of ideals for which a stable generating set was shown to exist (monomial ideals and ideals in two variables).

Our goal in this paper is two-fold.
First, we wish to show that the closure in $\cH_d^{(t)}$ of an ideal that has the stable division property is $p$-essentially normal for all $p>d$; in fact, only the approximate stable division property is required.
We only obtain this for $t>-3$. See Theorems \ref{thm:stab_EN} and \ref{thm:stab_EN2} for precise statements.
This improves on the result from \cite{Sha11} in that homogeneity is not required.
The key tools we use are the main results as well as some auxiliary results and techniques from \cite{FX13,FX15}, which say that principal ideals are $p$-essentially normal, together with the main result from \cite{Ken15}, which allows us to promote the result from principal ideals to general ideals.

Our second goal is to show that quasi homogeneous ideals in $\bC[x,y]$ have the stable division property with respect to any one of the $\cH_d^{(t)}$ norms; see  Theorem \ref{thm:stab_div_quasi}.
Theorems \ref{thm:stab_EN} (or \ref{thm:stab_EN2}) and \ref{thm:stab_div_quasi} combine to give a new proof that the closure of any quasi homogeneous ideal in $\cH_2^{(t)}$ is $p$-essentially normal for all $p>2$ and all $t \geq -2$.
This result was already obtained using different methods in \cite[Theorem 4.2]{DS11} and \cite[Corollary 1.3]{GZ13}.
In fact, using results from \cite{FX15}, one can easily get more --- we can show $p$-essential normality for every ideal in two variables; this is obtained in Theorem \ref{thm:two_dim}.
However, we think that the application of stable division in a new setting --- even to obtain a known result --- is an important development, and this urges us to continue to look for stable division (or lack of) in other classes of ideals.

\section{Stable division and essential normality in the Hilbert spaces $\cH_d^{(t)}$}

\begin{lemma}\label{lem:T}
Let $M_1, \ldots, M_k$ be linear subspaces of a Hilbert space $H$.
Let $M = M_1 + \ldots + M_k$ denote the algebraic sum of these spaces inside $H$, and let $\ol{M_1} \oplus \ldots \oplus \ol{M_k}$ denote the (disjoint) orthogonal sum formed by them.
Consider the map
\[
T : \ol{M_1} \oplus \ldots \oplus \ol{M_k} \longrightarrow H
\]
given by
\[
T(m_1, \ldots, m_k) = m_1 + \ldots + m_k.
\]
Then the following are equivalent:
\begin{enumerate}
\item T has closed range.
\item There exists a constant $C$ such that for every $h \in M$ and every $\epsilon > 0$, there are $m_1 \in M_1, \ldots,  m_k \in M_k$ satisfying
\be\label{eq:msum}
\|h - (m_1 + \ldots + m_k)\| \leq \epsilon ,
\ee
and
\be\label{eq:mnorm}
\sum \|m_i\|^2 \leq C \|h\|^2.
\ee
\end{enumerate}
In case that all the subspaces $M_i$ are closed, then (1) implies (2) with $\epsilon = 0$ as  well.
\end{lemma}

\begin{proof}
Let $K$ be the kernel of $T$, and consider the induced map
\[
\tilde{T} : G:= (\ol{M_1} \oplus \ldots \oplus \ol{M_k})/K \longrightarrow H.
\]
If $T$ has closed range, then $\tilde{T}$ is a bounded linear bijection of $G$ onto the Hilbert space $\ol{M}$, hence $\tilde{T}^{-1}$ is bounded.
Say $\|\tilde{T}^{-1}\| \leq c$. Given $h \in M$, let $g = \tilde{T}^{-1}(h)$. Then $\|g\| \leq c \|h\|$, which means that there is an element $(m_1, \ldots, m_k)$ in $\ol{M_1} \oplus \ldots \oplus \ol{M_k}$ satisfying (\ref{eq:msum}) and (\ref{eq:mnorm}) with $\epsilon = 0$ and any $C > c^2$.
In the case where the subspaces $M_i$ are not closed we may replace this tuple by $(m_1, \ldots, m_k)$ in $M_1 \oplus \ldots \oplus M_k$ to obtain (\ref{eq:msum}) and (\ref{eq:mnorm}) with $\epsilon > 0$ arbitrarily small and some constant $C > c^2$.

Conversely, assume that condition (2) above holds, and let $h \in \ol{M}$. Let $\{h^{(n)}\}$ be a sequence in $M$ converging to $h$ with $\|h^{(n)}\| \leq \|h\|$ for all $n$. For all $n$, we find $(m^{(n)}_1, \ldots, m^{(n)}_k) \in M_1 \oplus \ldots \oplus M_k$ satisfying
\[
\|h^{(n)} - (m_1^{(n)} + \ldots + m^{(n)}_k)\| \leq 1/n
\]
and
\[
\sum \|m^{(n)}_i\|^2 \leq C \|h\|^2.
\]
By weak compactness of the unit ball we find $(m_1, \ldots, m_k)$ in $\ol{M_1} \oplus \ldots \oplus \ol{M_k}$ satisfying (\ref{eq:mnorm}) and (\ref{eq:msum}) with $\epsilon = 0$. In particular, the range of $T$ is equal to $\ol{M}$, and is therefore closed.
\end{proof}

\begin{remark}\label{rem:notsquare}
Note that, given $k$, condition \eqref{eq:mnorm} is equivalent to
\begin{eqnarray*}
\sum \|m_i\| \leq C' \|h\|.
\end{eqnarray*}
\end{remark}

The relevance of Lemma \ref{lem:T} to the problem of stable division is that $\{f_1, \ldots, f_k\}$ is an approximate stable generating set for $I$ if and only if the subspaces $M_i = \langle f_i \rangle$ satisfy condition (2) of the lemma.

\begin{theorem}\label{thm:stab_EN}
Let $t> -3$ and let $I$ be an ideal in $\mb{C}[z]$.
Assume that $I$ has the approximate stable division property with respect to the $\cH_d^{(t)}$ norm.
Then the submodule $\ol{I}$ of $\cH_d^{(t)}$ is $p$-essentially normal for all $p>d$.
\end{theorem}

\begin{proof}
Let $\{f_1, \ldots, f_k\}$ be an approximate stable generating set for $I$, and for all $i=1,\ldots, k$ let $N_i$ denote the closure of the principal ideal $\langle f_i \rangle$.
By Theorem 1.1 in \cite{FX13} and Theorem 1.2 in \cite{FX15}, every one of the submodules $N_i \subseteq \cH_d^{(t)}$ is $p$-essentially normal for all $p>d$.

By Lemma \ref{lem:T} above, the algebraic sum $N_1 + \ldots + N_k$ is closed; it follows that it must be equal to the closure of $I$,
\[
\ol{I} = N_1 + \ldots + N_k.
\]
In the terminology of \cite{Ken15}, this means that $\ol{I}$ is $p$-essentially decomposable (meaning simply that $\ol{I}$ is the algebraic sum of closed $p$-essentially normal submodules of $\cH^{(t)}_d$).
By Theorem 3.3 of \cite{Ken15}, $\ol{I}$ is also $p$-essentially normal.

To be precise, \cite[Theorem 3.3]{Ken15} is stated for submodules of $H^2_d = \cH_d^{(-d)}$, but the proof of this result is a direct application of \cite[Theorem 4.4]{Arv07}, which is stated in the generality of abstract  Hilbert modules which are $p$-essentially normal, and therefore applies here.
\end{proof}

The following variant of Theorem \ref{thm:stab_EN} highlights somewhat different techniques.
Note that now the assumption is that there is stable division with respect to the $\cH_d^{(t+1)}$ norm, and the proof of the previous theorem does not seem to apply.

\begin{theorem}\label{thm:stab_EN2}
Let $t> -3$ and let $I$ be an ideal in $\mb{C}[z]$.
Assume that $I$ has the stable division property with respect to the $\cH_d^{(t+1)}$ norm.
Then the submodule $\ol{I}$ of $\cH_d^{(t)}$ is $p$-essentially normal for all $p>d$.
\end{theorem}

\begin{proof}
Denote $N = \ol{I\,}^t$ (meaning the closure of $I$ with respect to the $\cH_d^{(t)}$ norm).
We will apply the techniques developed for the proof of Theorem 1.1 in \cite{FX13}.
One of the standard approaches for showing that a closed submodule $N$ of $H \otimes \mb{C}^r$ is $p$-essentially normal is to show that
\[
S^*_i P_N - P_N S^*_i = S^*_i P_N - P_N S^*_i P_N = P_N^\perp S^*_i P_N \in \mathcal{L}^{2p} ,
\]
(where $P_N$ denotes the orthogonal projection of $H$ onto $N$ and $P_N^\perp = I - P_N$).
See, e.g., \cite[Proposition 4.2]{Arv07}.
To show that $P_N^\perp S^*_i P_N$ is in $\mathcal{L}^{2p}$, we will use the following result from \cite[Proposition 4.2]{FX13}, which we reformulate slightly for our needs.
Recall that $\cH_d^{(t)}$ is always contained in $\cH_d^{(t+1)}$.

\begin{proposition}[Fang-Xia]\label{prop:FX}
Let $I$ and $N$ be as above, let $T$ be a linear operator on $\cH_d^{(t)}$ and suppose that there is some $K$ such that
\[
\|Tg\|_t \leq K\|g\|_{t+1}
\]
for all $g \in I$. Then $T P_N \in \cL^{2p}$ for all $p>d$.
\end{proposition}
By Fang and Xia's proposition and the above remarks, it suffices to show that there is a constant $K$ such that
\[
\|P_N^\perp S^*_j h\|_t \leq K \|h\|_{t+1} \,\, , \,\, h \in I.
\]

Let $f_1, \ldots, f_k$ be a stable generating set for $I$ in the $\cH_d^{(t+1)}$ norm.
Denote by $Q^{(t)}_i$ the orthogonal projection onto the complement in $\cH_d^{(t)}$ of the ideal generated by $f_i$.
In the proof of \cite[Theorem 1.1]{FX13} (see p. 3007 there) it is shown that there is an integer $\ell_i$ and a constant $C_i$ such that
\be\label{eq:ineq_QSqf}
\|Q_i^{(t)}S^*_j g f_i \|_{t} \leq C_i \|g f_i \|_{t+1}
\ee
for all $g$ satisfying $\partial^\alpha g(0) = 0 $ when $|\alpha|\leq \ell_i$
(for the range $t \in (-3,-2]$ we require the proof of \cite[Theorem 1.2]{FX15}, see Proposition 4.4 --- in particular Equation (4.10) --- and Proposition 5.10 there).
Our goal now is to show that the inequality \eqref{eq:ineq_QSqf} holds for all $g\in\mathbb{C}[z]$, with a perhaps larger constant $C_i$.
This follows from the fact that ${\overline E_i}^{t+1}$ has finite codimension in ${\overline {\langle f_i \rangle}}^{t+1}$, together with Lemma \ref{lem:gen} in the appendix below, but we wish to give the argument in full detail.

First, we note that the ideal
\[
E_i = \{g f_i : \partial^\alpha g(0) = 0 \textrm{ for all } |\alpha|\leq \ell_i \}
\]
is finite codimensional in $\langle f_i \rangle$.
Let $R_i = \mbox{span}\{z^\alpha f_i : |\alpha| \leq \ell_i\}$. Since the equivalence classes $[z^\alpha f_i],\,|\alpha| \leq \ell_i$, form a basis for $\langle f_i \rangle/ E_i$, we have $\langle f_i \rangle = E_i \dot{+} R_i$, direct sum as linear subspaces.

Next, we claim that ${\overline E_i}^{t+1} \cap R_i = \{0\}$, where ${\overline E_i}^{t+1}$ is the closure of $E_i$ in the $\mathcal H_d^{(t+1)}$ norm.
To see this, fix $g f_i\in R_i$ for some polynomial $g$ with $\deg (g) \leq l_i$.
Let $z^{\beta_i}$ be the monomial of least degree, with respect to graded lexicographic ordering, that appears with a non-zero coefficient in the expression of $f_i$.
Likewise, let $z^\alpha$ be a monomial of least degree that appears with a non-zero coefficient in the expression of $g$.
The coefficient of $z^{\alpha + \beta_i}$ in the expression of $g f_i$ is then nonzero.
Since every monomial that appears in any polynomial in $E_i$ is strictly greater than $z^{\alpha + \beta_i}$ in the aforementioned ordering, and since monomials are orthogonal in $\cH_d^{(t+1)}$, we have that $E_i \subseteq \{\mathbb C z^{\alpha + \beta_i}\}^{\perp}$ (where $\{\mathbb C z^{\alpha + \beta_i}\}^{\perp}$ denotes the orthogonal complement in $\cH_d^{(t+1)}$ of the one dimensional subspace spanned by $z^{\alpha + \beta_i}$).
Thus
$$
{\overline E_i}^{t+1}\subseteq \{\mathbb C z^{\alpha + \beta_i}\}^{\perp}.
$$
But as $\langle g f_i, z^{\alpha + \beta_i} \rangle \neq 0$, our claim that ${\overline E_i}^{t+1} \cap R_i = \{0\}$ follows.

Let $M =  E_i$, $N = R_i$ considered as subspaces of $\cH_d^{(t+1)}$, let $K = \cH_d^{(t)}$, and let $T$ be the restriction of $Q^{(t)}S_j^*$ to $\bC[z]$.
Lemma \ref{lem:gen} then tells us that \eqref{eq:ineq_QSqf} implies that $Q^{(t)}S_j^*$ is bounded from $\langle f_i \rangle = E_i + R_i\subset \cH_d^{(t+1)}$ to $\cH_d^{(t)}$, meaning that \eqref{eq:ineq_QSqf} holds for all $g \in \bC[z]$, perhaps with a bigger constant $C_i$.

Now let $C = \max{C_i}$.
Since $P_N^\perp \leq Q^{(t)}_i$, then for every $g \in \bC[z]$, every $f_i$ in the stable generating set, and all $j$,
\bes
\|P_N^\perp S^*_j g f_i \|_{t} \leq C \|g f_i \|_{t+1} .
\ees

For $h \in I$,
let $g_1, \ldots, g_k \in \bC[z]$ satisfy  \eqref{eq:sum_app} and \eqref{eq:stab_app} with respect to the $\cH_d^{(t+1)}$ norm and with $\epsilon = 0$.
Then we have
\begin{align*}
\|P_N^\perp S^{*}_j h \|_{t} &\leq \sum_{i=1}^k \|P_N^\perp S^{*}_j g_i f_i \|_{t}\\ &\leq \sum_{i=1}^k C \|g_i f_i \|_{t+1}\\ &\leq AC \|h\|_{t+1} .
\end{align*}
By Proposition \ref{prop:FX}, $P_N^\perp S^*_j P_N$ is in $\cL^{2p}$ for all $p>d$.
By the remarks at the beginning of the proof, we are done.
\end{proof}

\begin{remark}
Here is another proof of the claim ${\overline E_i}^{t+1} \cap R_i = \{0\}$ that works in any analytic Hilbert module in the sense of \cite{CG} where monomials are not necessarily orthogonal.
Following \cite[Theorem 2.1.7(2)]{CG}, one could write
\be\label{eq:charac}
E_i = \{q\in \langle f_i \rangle : p(D)q|_{0} = 0\, \text{ for all } p\in E_{i0}\} ,
\ee
where $E_{i0}$ is the characteristic space of $E_i$ at $0$ and $p(D)$ denotes the differential operator $\sum_{\alpha}a_{\alpha}\frac{\partial^{\alpha}}{{\partial z}^\alpha}$ for any polynomial $p = \sum_{\alpha}a_{\alpha}z^\alpha$. This is because the characteristic spaces of $E_i$ and $\langle f_i \rangle$ are same at any $\lambda\neq 0$ and different at $\lambda = 0$. The first assertion follows from the fact that $E_i = \mathfrak M_0^{l_i+1}\langle f_i \rangle$, $\mathfrak M_0$ is the maximal ideal of $\mathbb C[z]$ at $0$,  and by repeated use of \cite[Proposition 1.3]{DG} while the last follows as at $\lambda = 0$, one observes that
$$
p(D)\{z^\alpha f_i\}|_0 = \frac{\partial^{\alpha}p}{{\partial z}^\alpha}(D)f_i|_0 ,
$$
and hence
$$
E_{i0} = \{q\in\mathbb C[z] : \frac{\partial^{\alpha}q}{{\partial z}^\alpha}\in\langle f_i \rangle_{0} \mbox{~for~all~}\alpha \mbox{~such~that~}|\alpha| = l_i+1\} ,
$$
(where $\langle f_i \rangle_{0}$ is the characteristic space of $\langle f_i \rangle$ at $0$).
Having establsihed \eqref{eq:charac}, we obtain that
$$
{\overline E_i}^{t+1} \subset \{f\in \cH_d^{(t+1)} : p(D)f|_{\lambda} = 0,\, \mbox{~for~all~}p\in E_{i0}\},
$$
and this combined with \eqref{eq:charac} ensures that ${\overline E_i}^{t+1} \cap R_i = \{0\}$.
\end{remark}

\section{Stable division for quasi homogeneous ideals in $\mb{C}[x,y]$}

Theorems \ref{thm:stab_EN} and \ref{thm:stab_EN2} motivate us to find new examples of modules which have the (approximate) stable division property.
We will show that quasi homogeneous ideals in $\mb{C}[x,y]$ have the stable division property (it is convenient to use the notation $\mb{C}[x,y]$ for the case $d=2$).
Our discussion is an improvement to \cite[Section 2.2]{Sha11}, where it was proved that {\em homogeneous} ideals in $\mb{C}[x,y]$ have the stable division property.

Before proceeding, we record the following proposition, which shows that when the generating set consists of quasi homogeneous polynomials, the stable division and the approximate stable division properties are the same.
Since this result is not needed we omit the simple proof.

\begin{proposition}
Let $M$ be an ${\bf n}$-quasi homogeneous module in $\mb{C}[z] \otimes \mb{C}^r$.
Then an approximate stable generating set consisting of ${\bf n}$-quasi homogeneous polynomials is a stable generating set.
\end{proposition}

\begin{lemma}\label{lem:stab_div_diff_norms}
Let $M$ be a quasi homogeneous submodule in $\mb{C}[z] \otimes \mb{C}^r$ and let $t>-d$.
Then a generating set for $M$ consisting of quasi homogeneous polynomials is a stable generating set with respect to the $H^2_d$ norm if and only if it is a stable generating set with respect to the $\cH_d^{(t)}$ norm.
\end{lemma}

\begin{proof}
Suppose that $M$ be an ${\bf n}$-quasi homogeneous submodule, where ${\bf n} = (n_1, \ldots, n_d)$.
Denote $n = \max_{1 \leq i \leq d} n_i$.
Let $f_1, \ldots, f_k$ be a generating set for $M$ such that every $f_i$ is quasi homogeneous.
Let $h$ be a quasi homogeneous element of $M$, $deg_{\bf n}(h) = m$.
If $z^\alpha$ is some monomial appearing in $h$, then $|\alpha| = \alpha \cdot {\bf 1} \leq \alpha \cdot {\bf n} = deg_{\bf n}(z^\alpha) = m$;
while on the other hand $m = \alpha \cdot {\bf n} \leq n |\alpha|$.
Thus $h = \sum_{i=\floor{m/n}}^m h_i$ where every $h_i$ is homogeneous, so by (\ref{eq:equiv_norm})
\[
c_{m,t}\|h\|^2_{H^2_d} \leq \|h\|^2_t \leq c_{\floor{m/n},t}\|h\|^2_{H^2_d} .
\]
Now suppose that $h = \sum_{j=1}^k a_j f_j$ and that $deg_{\bf n}(a_j f_j) = m$ for all $j$. Assume that $\sum_j \|a_j f_j\|^2_{H^2_d} \leq C \|h\|^2_{H^2_d}$. Then
\begin{align*}
\sum_j \|a_j f_j\|_t^2 &\leq \sum_j c_{\floor{m/n},t}\|a_j f_j\|^2_{H^2_d}\\  &\leq C c_{\floor{m/n},t}\|h\|^2_{H^2_d} \\ &\leq C \frac{c_{\floor{m/n},t}}{c_{m,t}} \|h\|^2_t .
\end{align*}
Since
\be\label{eq:NEEDS_CHECKING}
\lim_{m \rightarrow \infty} \frac{c_{\floor{m/n},t}}{c_{m,t}} \in [1, \infty),
\ee
we see that stable division with respect to the $H^2_d$ norm implies stable division with respect to the $\cH_d^{(t)}$ norm.
The converse is proved in the same way.
\end{proof}

We now recall some facts from computational algebraic geometry.
A standard reference for this is \cite{CLS92}. Let ${\bf n} \in \mb{N}^2$ be a weight vector.
We fix an ordering on the monomials in $\mb{C}[x,y]$ that is determined by ${\bf n}$: if $a$ and $b$ are non-zero scalars, then we say that $b x^m y^n$ is {\em greater than} $a x^k y^l$, denoted $a x^k y^l < b x^m y^n$, if and only if
\[
(k n_1+ l n_2 <  m n_1 + n n_2) \textrm{ or } (k n_1+ l n_2 =  m n_1 + n n_2 \textrm{ and } k < m).
\]
The {\em leading term} of a polynomial $f$, denoted $LT(f)$, is the largest monomial appearing in $f$. We say that $b x^m y^n$ is {\em divisible} by $a x^k y^l$ (where $a,b \in \bC \setminus \{0\}$) if and only if $k \leq m$ and $l \leq n$, and then we have $(b x^m y^n) / (a x^k y^l) := (b/a) x^{m-k}y^{n-l}$.

Let us remind the reader of the {\em division algorithm} for polynomials in several variables (we need it only in two variables).
Given $h, f_1, \ldots, f_k \in \mb{C}[x,y]$ we may divide $h$ by $f_1, \ldots, f_k$ with remainder, meaning that we find $a_1, \ldots, a_k, r \in \mb{C}[x,y]$ (with $LT(r)$ not divisible by any $LT(f_i)$) such that
\be\label{eq:div_w_r}
h = \sum a_i f_i + r .
\ee
The decomposition \eqref{eq:div_w_r} is obtained using the following algorithm.
We start with
\[
r = a_1 = \ldots = a_k = 0 ,
\]
and define a temporary polynomial $p$ which is set at the start to $p=h$.
We now start iterating over the terms of $p$ in decreasing order.
If $LT(p)$ is divisible by $LT(f_i)$ for some $i$, we replace  $a_i$ by $a_i + LT(p)/LT(f_i)$ and replace $p$ with $p - \left(LT(p)/LT(f_i)\right) f_i$.
In principal one is free to choose which $i$ to use, but we will always choose the {maximal $i$ possible}.
If $LT(p)$ is not divisible by any of the $LT(f_i)$s, we put $r = r + LT(p)$ and replace $p$ with $p - LT(p)$.
We continue this way until $p=0$.

Even if $h$ is in the ideal generated by $f_1, \ldots, f_k$, the above division algorithm might terminate with a non-trivial remainder.
However, for every ideal $I$ there exists a {\em Groebner basis} $\{f_1, \ldots, f_k\}$, which is a generating set for $I$ with the property that the above algorithm, when run with $h \in I$ and $f_1, \ldots, f_k$ as input, terminates with $r = 0$ \cite[Corollary 2, p. 82]{CLS92}.

\begin{lemma}\label{lem:stab_div_r}
Let ${\bf n} = (n_1, n_2)$ be a weight, and let $f_1, \ldots, f_k$ be ${\bf n}$-quasi homogeneous polynomials in $\mb{C}[x,y]$ with $deg_{\bf n}(f_i) = m$ for $i=1, \ldots, k$.
There exists a constant $C$ such that for every ${\bf n}$-quasi homogeneous polynomial $h \in \mb{C}[x,y]$ there are $a_1, \ldots, a_k, r \in \mb{C}[x,y]$ such that
\[
h = \sum_{i=1}^k a_i f_i + r
\]
with
\be\label{eq:stab_div_r}
\sum_{i=1}^k \|a_i f_i\|^2_{H^2_d} \leq C\left( \|h\|^2_{H^2_d} + \|r\|^2_{H^2_d} \right) .
\ee
In fact, the $a_i$s and $r$ can be found by running the division algorithm.
\end{lemma}
\begin{remark}
Note that we will use the same weight ${\bf n}$ determining quasi homogeneity as the weight determining the monomial ordering.
In fact, we choose the order on monomials according to type of quasi homogeneous polynomials we wish to deal with.
In the case that we are dealing with homogeneous ideals then the monomial order is the graded lexicographic order, as in \cite[Lemma 2.3]{Sha11} --- on which the following proof is modeled.
\end{remark}
\begin{proof}
In the proof we assume without loss of generality that $n_1 \geq n_2$.
We will show, by induction on $k$, that there is a constant $C$ such that when running the division algorithm in an appropriate manner to divide $h$ by $f_1, \ldots, f_k$, one has the estimate (\ref{eq:stab_div_r}).

If $k=1$, then this is obvious from the triangle inequality, so assume $k>1$.
Suppose that $LT(f_1) > LT(f_2) > \ldots > LT(f_k)$. Let
\[
f_i = \sum_{j}a_{ij}x^{\frac{m-n_2j}{n_1}}y^j
\]
where the sum is on all $j$ starting from $j_i$ until the last such $j$ for which there is some non-negative integer $l$ satisfying $m = l n_1 + j n_2$.
Because of our assumptions we have $j_1 < j_2 < \ldots < j_k$.

By compactness considerations, we may assume that $deg_{\bf n}(h) = n > 4m$.
Now initiate the algorithm as in the description above by setting $p = h$.
At a certain iteration of the division algorithm we have
\[
LT(p) = b_t x^{\frac{n - tn_2}{n_1}}y^t.
\]
When $t<j_1$, then $LT(p)$ is not divisible by any $LT(f_i)$, so we move $LT(p)$ to the remainder.
This happens at most $j_1-1$ iterations.
We will use $f_1$ to divide $p$ at most $j_2 - j_1$ iterations, only when $j_1 \leq t < j_2$.
After $t$ becomes greater than $j_2$ we will never have to use $f_1$ again, and we can use the inductive hypothesis.
All there is to check is that there is some constant $C$, independent of $h$, such that $\|p\|^2 \leq C \|h\|^2$ and $\|a_1 f_1 \|^2 \leq C\|h\|^2$ after the last iteration when the algorithm used $f_1$.

When $t<j_1$, $\|p\|$ only decreases and $a_i$ does not change.
Then there are at most $j_2 - j_1$ iterations in which $a_1 f_1$ is modified to $(a_1 + LT(p)/LT(f_1))f_1$ and $p$ is modified to $p - LT(p)/LT(f_1) f_1$.
Thus we will be done once we show that there is some constant $C$ such that
\[
\|LT(p)/LT(f_1) f_1\|^2 \leq C\|p\|^2.
\]
We compute
\[
\frac{LT(p)}{LT(f_1)} f_1 = \frac{b_t}{a_{1j_1}} \sum a_{1j}x^{\frac{n-n_2(t+j-j_1)}{n_1}}y^{t+j-j_1}.
\]
Therefore
\[
\left\|\frac{LT(p)}{LT(f_1)} f_1\right\|^2 = \left|\frac{b_t}{a_{1j_1}}\right|^2 \sum |a_{1j}|^2 \frac{(\frac{n-n_2(t+j-j_1)}{n_1})!(t+j-j_1)!}{(\frac{n+ (n_1 - n_2)(t+j-j_1)}{n_1})!} .
\]
Now consider the integer $s = t+j-j_1$ occurring in the above expression.
Because $j$ runs from $j_1$ to the highest power of $y$ appearing in $f_1$, we have that $t \leq s \leq 2m$. Thus
\begin{align*}
\frac{(\frac{n-n_2s}{n_1})!s!}{(\frac{n+ (n_1 - n_2)s}{n_1})!} & \leq \frac{(\frac{n-n_2s}{n_1})!(2m)!}{(\frac{n+ (n_1 - n_2)t}{n_1})!}\\ &\leq \frac{(\frac{n-n_2t}{n_1})!(2m)!}{(\frac{n+ (n_1 - n_2)t}{n_1})!} \\
& \leq (2m)! \frac{(\frac{n-n_2t}{n_1})!t!}{(\frac{n+ (n_1 - n_2)t}{n_1})!}\\ & = \frac{(2m)!}{|b_t|^2} \|LT(p)\|^2.
\end{align*}
Thus we have the bound
\[
\|LT(p)/LT(f_1) f_1\|^2 \leq \frac{(2m)!}{|a_{1j_1}|^2} \sum |a_{1j}|^2 \|p\|^2,
\]
and since $C = \frac{(2m)!}{|a_{1j_1}|^2} \sum |a_{1j}|^2$ is independent of $h$, this completes the proof.
\end{proof}

\begin{theorem}\label{thm:stab_div_quasi}
Let $I$ be an ${\bf n}$-quasi homogeneous ideal in $\mb{C}[x,y]$.
For all $t \geq -2$, $I$ has the stable division property in the $\cH_2^{(t)}$ norm.
If $\{f_1, \ldots, f_k\}$ is a Groebner basis for $I$ consisting of ${\bf n}$-quasi homogeneous polynomials of the same degree, then it is a stable generating set.
\end{theorem}

\begin{proof}

One can always construct a Groebner basis for $I$ consisting of ${\bf n}$-quasi homogeneous polynomials (Exercise 3 on p. 377,  \cite{CLS92}).
By multiplying basis elements by monomials one then obtains a Groebner basis consisting of ${\bf n}$-quasi homogeneous polynomials of the same degree $m$, for the finite co-dimensional subideal $I_m \oplus I_{m+1} \oplus \ldots$ where $m$ is sufficiently large.
Thus, we concentrate on proving the second assertion.

Let $\{f_1, \ldots, f_k\}$ be a Groebner basis for $I$ consisting of ${\bf n}$-quasi homogeneous polynomials of the same degree $m$.
By Lemma \ref{lem:stab_div_diff_norms}, it suffices prove the assertion for $t=-2$, that is, to show that $\{f_1, \ldots, f_k\}$ is a stable generating set for $I$ in the $H^2_2$ norm.

Let $h \in I$.
We may assume $h$ is quasi homogeneous, otherwise decompose it into an orthogonal sum of quasi homogeneous polynomials.
We divide $h$ by $f_1, \ldots, f_k$ and obtain $h = \sum a_i f_ i + r$ with (\ref{eq:stab_div_r}).
But since $\{f_1, \ldots, f_k\}$ is a Groebner basis, $r = 0$ (Corollary 2 on p. 82,  \cite{CLS92}).
Thus $\{f_1, \ldots, f_k\}$ is a stable generating set, and $I$ has the stable division property.
\end{proof}

Putting together the previous theorem with either Theorem \ref{thm:stab_EN} or Theorem \ref{thm:stab_EN2} we obtain the following corollary.

\begin{corollary}\label{cor:ess_norm_quasi}
Let $t\geq -2$ and let $I$ be a quasi homogeneous ideal in $\mb{C}[x,y]$.
Then the submodule $\ol{I}$ of $\cH_2^{(t)}$ is $p$-essentially normal for all $p>2$.
\end{corollary}

The above corollary was proved by Douglas and Sarkar by completely different methods, see \cite[Theorem 4.2]{DS11} (another proof appeared in \cite[Corollary 1.3]{GZ13}).
In fact, using Fang and Xia's result we may prove a stronger result, which applies to not necessarily quasi homogeneous ideals.
This has been observed in \cite[Corollary 3.1]{GZ13} for the case $t>-2$.

\begin{theorem}\label{thm:two_dim}
Let $I$ be an ideal in $\mb{C}[x,y]$.
Then for all $t \geq -2$ and all $p>2$, the submodule $\ol{I}$ of $\cH_2^{(t)}$ is $p$-essentially normal.
\end{theorem}
\begin{proof}
Let $p = \operatorname{gcd}(I)$.
Then $I= p J$, where $J$ is a finite codimensional ideal in $\mb{C}[x,y]$ \cite[Lemma 2.2.9, p.28]{CG}.
Thus, $I$ is finite codimensional in $\langle p \rangle$, so $\ol{I}$ is finite codimensional in $\ol{\langle p \rangle }$.
By \cite[Theorem 1.1]{FX13} and \cite[Corollary 1.3]{FX15}, $\ol{\langle p \rangle}$ is $p$-essentially normal for all $p>2$, and from finite codimensionality it follows that $\ol{I}$ is, too.
\end{proof}

\begin{remark}
The factorization of $I$ as $pJ$ above is known as the {\em Beurling form} of $I$.
The idea of using the Beurling form to reduce the essential normality of an ideal to that of a principle ideal has appeared several places in the literature (for example in \cite{GW08,GZ13}) but we could not trace its precise origin.
Admittedly, the above proof is much shorter than the proof we provide for Corollary \ref{cor:ess_norm_quasi}.
On the other hand, it does require implicitly some results in algebraic geometry that are less intuitive than what is used in the stable division proof.
\end{remark}

Since the conjectures of Arveson and Douglas treat also submodules with multiplicity, it is natural to ask whether Lemma \ref{lem:stab_div_r} extends to vector valued modules.
The following example shows that the lemma fails in the vector valued case.
However, we will also show that the following example is {\em not} a counter example to the vector valued version of the first assertion of Theorem \ref{thm:stab_div_quasi}.

\begin{example}
Consider the module generated by $f_1 = (x, 0, y)$, $f_2 = (0,x,y)$, and let $h = (xy^n, -xy^n, 0)$.
Then
\[
h = y^nf_1 - y^n f_2 ,
\]
and there is no other way to write $h$ as an element in the module spanned by $f_1$ and $f_2$.
Therefore the output of the division algorithm is $a_1 = -a_2 = y^n$.
However $\|h\|^2 \sim 1/n$, while $\|a_if_i\|^2 \sim 1$.
Thus, the conclusion of Lemma \ref{lem:stab_div_r} fails.

On the other hand, we now show that the basis $\{f_1 - f_2, f_2\}$
is a stable generating set for the module spanned by $f_1$ and $f_2$.
To see this, consider $h = pf_1 + q f_2 = (px, qx, (p+q)y)$, where $p,q \in \bC[x,y]$.
Then
\[
\|h\|^2 = \|px\|^2 + \|qx\|^2 + \|(p+q)y\|^2 .
\]
Write $h = p(f_1 - f_2) + (p+q)f_2$.
Then we compute the norm of each term:
\[
\|p(f_1 - f_2)\|^2 = \|(px, -px, 0)\|^2 = 2\|px\|^2 \leq 2 \|h\|^2 ,
\]
and
\begin{align*}
\|(p+q)f_2\| &= \|(0,(p+q)x,(p+q)y)\| \\
&\leq \|px\| + \|qx\| + \|(p+q)y\| \leq \sqrt{3} \|h\|.
\end{align*}
Thus $\{f_1 - f_2, f_2\}$ is a stable generating set.
\end{example}

We stress that it is still an open problem whether there exists an ideal that does not have the stable division property.
Examples as above show why it is so hard to settle the problem of whether or not every (homogeneous) ideal has the stable division property.
On the one hand, it is not too hard to cook up an ideal with a generating set which is not a stable generating set.
On the other hand, in all examples that we know, after making a few changes to the generating set it becomes a stable generating set.

\section{Appendix}

In the proof of Theorem \ref{thm:stab_EN2} we required the following elementary result, the proof of which we include for completeness.

\begin{lemma}\label{lem:gen} Let $M$ and $N$ be two linear subspaces in a Hilbert space $H$ such that $\overline{M}\cap N =\{0\}$.
Let $T$ be a linear operator defined on $M + N$ mapping to another Hilbert space, say $K$, and assume that
(i) $T$ is bounded on $M$,
(ii) $N$ is finite dimensional.
Then $T$ is bounded on $M+N$.
\end{lemma}
\begin{proof}
By induction, it suffices to prove for the case that $N$ is one dimensional, say $N = \operatorname{span} \{v\}$ where $v$ is a unit vector. We will show in fact that $T$ extends to a bounded operator from $\overline{M} + N$ into $K$.
Now, $T$ extends to a bounded operator on $\overline{M}$.
Moreover, $T$ is bounded on $N$ because $N$ is finite dimensional.
Denote $C = \max\{\|T\big|_N\|, \|T\big|_M\|\}$. Since $v \notin \overline{M}$, $v$ has a positive angle from $\overline{M}$ in the sense that $c:=\sup \{|\langle v, x \rangle | : x \in M, \|x\|=1\} < 1$.
It follows that for $m \in M$ and $n \in N$,
\[
\|m + n\|^2 \geq \|m\|^2 + \|n\|^2 - 2 c \|m\| \|n\| \geq (1-c)(\|m\|^2+ \|n\|^2) .
\]
Thus, for $m \in M$ and $n \in N$, we have
\begin{eqnarray*}\|T (m + n) \| &\leq & \|T m\| + \|T n\| \leq C (\|m\| + \|n\|)\\ &\leq & C \sqrt{2} (\|m\|^2 + \|n\|^2)^{1/2}\\ & \leq & C \sqrt{2} (1-c)^{-1/2} \|m + n\|.\end{eqnarray*}
\end{proof}

\subsection*{Acknowledgement} The authors are grateful to an anonymous referee for spotting a couple of mistakes in a previous version of this paper, and for several other helpful remarks that improved our presentation.
The authors also wish to thank Guy Salomon for providing useful feedback.


\bibliographystyle{amsplain}

\end{document}